\documentclass[a4paper]{amsart}
\usepackage{amssymb}
\usepackage{hyperref}
\usepackage{stmaryrd}

\newtheorem{theorem}{Theorem}[section]
\newtheorem{lemma}[theorem]{Lemma}

\newtheorem{proposition}[theorem]{Proposition}
\newtheorem{conjecture}[theorem]{Conjecture}

\theoremstyle{definition}

\newtheorem{hypothesis}[theorem]{Hypothesis}

\newtheorem{question}[theorem]{Question}
\newtheorem*{notation*}{Notation}
\newtheorem*{problem*}{Problem}

\newcommand{\Aut}{\mathop{\mathrm{Aut}}}
\newcommand{\Cay}{\mathop{\mathrm{Cay}}}
\newcommand{\Sym}{\mathop{\mathrm{Sym}}}
\newcommand{\Alt}{\mathop{\mathrm{Alt}}}
\def\K{{\rm K}}
\def\C{{\rm C}}
\def\Z{{\rm Z}}
\newcommand{\ZZ}{\mathbb{Z}}

\newcommand{\SL}{\mathrm{SL}}
\newcommand{\PGL}{\mathrm{PGL}} 
\newcommand{\GL}{\mathrm{GL}} 
\newcommand{\PSL}{\mathrm{PSL}}
\newcommand{\PSU}{\mathrm{PSU}}
\newcommand{\PSp}{\mathrm{PSp}}
\newcommand{\PSigmaL}{\mathrm{P\Sigma L}} 
\newcommand{\PGammaL}{\mathrm{P\Gamma L}}
\newcommand{\POmega}{\mathrm{P\Omega}}

\def\J{{\rm J}}
\def\M{{\rm M}}

\newcommand{\norml}{\trianglelefteqslant}

\renewcommand{\geq}{\geqslant}
\renewcommand{\leq}{\leqslant}

\begin{document}
\title[Cubic arc-transitive $k$-circulants]{Cubic arc-transitive $k$-circulants}

\author[Michael Giudici, Istv\'{a}n Kov\'{a}cs, Cai Heng Li, Gabriel Verret]{Michael Giudici, Istv\'{a}n Kov\'{a}cs, Cai Heng Li, Gabriel Verret}

 \address{Michael Giudici, Cai Heng Li and Gabriel Verret, 
\newline\indent Centre for the Mathematics of Symmetry and Computation, 
\newline\indent The University of Western Australia, 
\newline\indent 35 Stirling Highway, Crawley, WA 6009, Australia.} 

 \address{Istv\'{a}n Kov\'{a}cs and Gabriel Verret,
\newline\indent IAM and FAMNIT, University of Primorska, Glagolja\v{s}ka 8, SI-6000 Koper, Slovenia.\newline}

\email{Michael.Giudici@uwa.edu.au}
\email{Istvan.Kovacs@upr.si}
\email{Cai.Heng.Li@uwa.edu.au}
\email{Gabriel.Verret@uwa.edu.au}

\thanks{This research was supported by the Australian Research Council grants DE130101001 and DP150101066. The second author was supported in part by the Slovenian Research  Agency (research program P1-0285 and research projects N1-0032,  N1-0038, J1-5433, and J1-6720) and would also like to thank The University of Western Australia for its hospitality.}



\begin{abstract}
For an integer $k\geq 1$, a graph is called a \emph{$k$-circulant} if its automorphism group contains a cyclic semiregular subgroup with $k$ orbits on the vertices. We show that, if $k$ is even, there exist infinitely many cubic arc-transitive $k$-circulants. We conjecture that, if $k$ is odd, then  a cubic arc-transitive $k$-circulant has order at most $6k^2$. Our main result is a proof of this conjecture when $k$ is squarefree and coprime to $6$.
\end{abstract}
\maketitle
\section{Introduction}
All graphs in this paper are finite, simple and connected. A permutation group is called \emph{semiregular} if its only element fixing a point is the identity. For an integer $k\geq 1$, a graph is called a \emph{$k$-circulant} if its automorphism group contains a cyclic semiregular subgroup with $k$ orbits on the vertices.

Clearly, every graph is a $k$-circulant for some $k$, for example if $k$ is the order of the graph. Moving beyond this trivial observation is often quite difficult: whether every vertex-transitive graph is a $k$-circulant for some other $k$ is a famous open problem (see \cite{Cameron,maru81}). This question has been settled in the affirmative for graphs of valency at most four \cite{Dobson,MaruScap}.

On the other hand, studying $k$-circulants for fixed $k$ often yields interesting results. For example, $1$-circulants, usually called simply \emph{circulants}, are exactly Cayley graphs on cyclic groups. These graphs have been intensively studied. The family of $2$-circulants (sometimes called \emph{bicirculants}) has also attracted some attention.

In many cases, additional symmetry conditions are imposed on the graphs. In particular, cubic arc-transitive $k$-circulants have been the focus of some recent investigation.  (A graph is called \emph{arc-transitive} if its automorphism group acts transitively on ordered pairs of adjacent vertices. A graph is \emph{cubic} if each of its vertices has degree $3$.) 
It is a rather easy exercise to show that a cubic arc-transitive circulant is isomorphic to either $\K_4$ or $\K_{3,3}$.  The classification of cubic arc-transitive bicirculants can be deduced from \cite{FGW,MP,pisanski}, while cubic arc-transitive $k$-circulants for $k\in\{3,4,5\}$ are classified in \cite{Pentacirculant,Tricirculant}.  Rather than describe these classifications in detail, we would simply like to point out one striking feature: for $k=2$ or $k=4$, there exist infinitely many cubic arc-transitive $k$-circulants, whereas for $k\in\{1,3,5\}$, there are only finitely many. This immediately suggests the following question.

\begin{question}\label{mainquestion}
Given a positive integer $k$, does there exist infinitely many cubic arc-transitive $k$-circulants?
\end{question}

Investigating this question is the main topic of this paper. Our first result is that the answer is positive when $k$ is even.

\begin{proposition}\label{theorem:even}
If $k$ is an even positive integer, then there exist infinitely many cubic arc-transitive $k$-circulants.
\end{proposition}

We then turn our attention to the case when $k$ is odd. We are unable to settle Question~\ref{mainquestion} in full generality in this case, but we prove the following, which is our main result.

\begin{theorem}\label{theorem:squarefree}
If $k$ is a squarefree positive integer coprime to $6$, then  a cubic arc-transitive $k$-circulant has order at most $6k^2$.
\end{theorem}

We would like to note that the methods used in the proof of Theorem~\ref{theorem:squarefree} can, with some more effort, yield a complete classification of cubic arc-transitive $k$-circulants, for $k$ squarefree coprime to $6$. We also show that the bound of $6k^2$ in Theorem~\ref{theorem:squarefree} is best possible.

\begin{proposition}\label{prop:construction}
If $k$ is an odd positive integer, then there exists a cubic arc-transitive $k$-circulant of order $6k^2$
\end{proposition}

Finally, in view of Theorem~\ref{theorem:squarefree}, Proposition~\ref{prop:construction} and computational evidence gathered from the census of cubic arc-transitive graphs of order at most $10 000$~\cite{Conder10000,ConderDob}, we would like to propose the following conjecture which would completely settle Question~\ref{mainquestion}.
\begin{conjecture}\label{mainConjecture}
If $k$ is an odd positive integer, then  a cubic arc-transitive $k$-circulant has order at most $6k^2$.
\end{conjecture}

\section{Proof of Proposition~\ref{theorem:even}}
Let $k=2m$ be an even positive integer and let $p$ be a prime with $p\equiv 1\pmod 3$ and $p \nmid m$. Let  
$$G=\langle u,v,w,x\mid u^m,v^m,w^p,x^2,[u,v],[u,w],[v,w],u^xu,v^xv,w^xw\rangle.$$
In other words, $G$ is the generalised dihedral group on the abelian group $A:=\langle u,v,w\rangle\cong \ZZ_m^2 \times \ZZ_p$. Since $p\equiv 1\pmod 3$, there exists a positive integer $\alpha$ satisfying $\alpha^2+\alpha+1 \equiv 0  \pmod p$. 

Let $y$ be the automorphism of $G$ satisfying $u^y=v$, $v^y=u^{-1}v^{-1}$, $w^y=w^\alpha$,  and $x^y=x$, and let $s=uwx$. Note that $y$ has order $3$ while $s$ is an involution.

Let $S=\{s,s^y,s^{y^2}\}$, let $R=\langle S \rangle$ and let $\Gamma$ be the Cayley graph $\Cay(R,S)$. It is clear that $\Gamma$ is connected and, since $s^y \ne s$, we have $|S|=3$  and thus $\Gamma$ is cubic. Since $y$ has order $3$, $S$ is the orbit of $s$ under $\langle y \rangle$ which implies that $y$ induces an automorphism of $\Gamma$ fixing the identity but acting transitively on $S$. In particular, $\Gamma$ is arc-transitive.

Let $a=ss^y$ and $b=ss^{y^2}$. Note that $a,b\in A$ thus $a^s=a^x=a^{-1}$ and $b^s=b^x=b^{-1}$. In particular, $R = \langle a, b,s \rangle = \langle a,b\rangle \rtimes \langle s \rangle$. Now, a simple calculation yields that $a=u v^{-1} w^{1-\alpha}$ and $b=u^2 v w^{1-\alpha^2}$. In particular,  $a^m = w^{(1-\alpha)m}$. Since $p$ divides neither $m$ nor $1-\alpha$, it follows that $w\in \langle a\rangle$. This implies that $\langle a,b\rangle= \langle uv^{-1},u^2v \rangle  \times \langle w \rangle= \langle uv^{-1},u^3 \rangle\times \langle w \rangle$. In particular, $|A:\langle a,b\rangle|=3$  if $3$ divides $m$ and $\langle a,b\rangle=A$ otherwise. Similarly,  $|G:R|=3$  if $3$ divides $m$ and $G=R$ otherwise.

Let $C$ be the group generated by $u^3 w$ if $3$ divides $m$ and by $uw$ otherwise. Note that $C\leq\langle a,b\rangle\leq R$ and thus $C$ is a semiregular subgroup of $G$. Since $p$ is coprime to $m$, $C$ has order $\frac{mp}{3}$ if $3$ divides $m$, and $mp$ otherwise. In both cases, $C$ has $2m=k$ orbits on the vertices of $\Gamma$ and thus $\Gamma$ is a $k$-circulant. To conclude the proof, it suffices to note that there are infinitely many primes $p$ with $p\equiv 1\pmod 3,$ and $p \nmid m$.

\section{Proof of Proposition~\ref{prop:construction}}
Let $$R=\langle u,v,x,y \mid u^k,v^k,x^3,y^2,[u,v],u^xv^{-1},v^xvu,u^yv,v^yu,x^yx \rangle.$$
Note that $R=\langle u,v\rangle\rtimes\langle x,y\rangle\cong\ZZ_k^2\rtimes\Sym(3)$. Let 
$$
\delta=
\begin{cases}
1, \textrm{ if } 3 \textrm{ divides } k,\\
0, \textrm{ otherwise,}\\
\end{cases}
$$
and let $\sigma$ be the automorphism of $R$ such that $u^\sigma=u^{-1}v^{-1}$, $v^\sigma=u$,  $x^\sigma=(uv)^\delta x$, $y^\sigma=x^2y$, and let $G=R\rtimes\langle\sigma\rangle$. Note that one needs to check that $\sigma$ is indeed an automorphism of $R$. This is obvious when $\delta=0$ as, in this case, $\sigma$ acts on $R$ as conjugation by $x^{-1}$. Furthermore, even when $\delta=1$, $\sigma$ still acts on $\langle u,v\rangle\rtimes\langle y\rangle$ as conjugation by $x^{-1}$, hence one needs only check that the relations involving $x$ are preserved by $\sigma$. This is a straightforward computation. For example:
$$
(u^\sigma)^{x^\sigma}=(u^{-1}v^{-1})^{(uv)^\delta x}=(u^{-1}v^{-1})^x=u=v^\sigma.
$$
It is also easy to check that $x^{\sigma^3}=x$ and thus $\sigma$ has order $3$. Let $s=uvy$, let $S=\{s,s^\sigma,s^{\sigma^2}\}$, and let $\Gamma=\Cay(R,S)$. Note that $s$ is an involution and $s^\sigma=v^{-1}x^2y \ne s$ hence $S$ is an inverse-closed set of cardinality $3$ and $\Gamma$ is a cubic graph of order $|R|=6k^2$.  Since $S$ is a $\langle \sigma\rangle$-orbit, $\sigma$ induces an automorphism of $\Gamma$ fixing the identity but acting transitively on $S$. In particular, $\Gamma$ is $G$-arc-transitive.

We now show that $\Gamma$ is connected or, equivalently, that $S$ generates $R$. Recall that $s^\sigma=v^{-1}x^2y$. An easy computation yields $s^{\sigma^2}=u^{\delta-1}xy$. A more elaborate computation then yields that $ss^\sigma ss^{\sigma^2}=u^{3-\delta}\in\langle S\rangle$. From the definition of $\delta$, it follows that $u\in \langle S\rangle$, but then $u^s=v^{-1}\in\langle S\rangle$ and we easily conclude that $S$ generates $R$. It remains to exhibit a semiregular cyclic subgroup of $G$ of order $6k$.

Suppose first that $3$ does not divide $k$ and  let $C=\langle uv^{-1}yx\sigma\rangle$.  An easy computation shows that $uv^{-1}$, $y$ and $x\sigma$ pairwise commute and have orders $k$, $2$ and $3$, respectively. Since $k$ is coprime to $6$, it follows that $|C|=6k$. It remains to show that $C$ is semiregular. Since $|G_v|=3$, it suffices to show that $\langle x\sigma\rangle$ is semiregular but, in fact, $\langle x\sigma\rangle$ is central in $G$ and thus must intersect every point-stabiliser trivially.

Suppose now that $3$ divides $k$ and let $C=\langle y\sigma\rangle$. As $y\sigma$ is not contained in $R$ which is a normal subgroup of index $3$ in $G$, we get that $|y\sigma|=3|(y\sigma)^3|$. A computation yields that $(y\sigma)^3=v^{-1}xy$. Similarly,  $(y\sigma)^3$ is not contained in $\langle u,v\rangle\rtimes\langle x\rangle$ which  is a normal subgroup of index $2$ in $R$, and thus $|(y\sigma)^3|=2|(y\sigma)^6|$. Now, $(y\sigma)^6=(v^{-1}xy)^2=v^{-2}u^{-1}$ and $|v^{-2}u^{-1}|=k$ and thus $|C|=|y\sigma|=6k$. It remains to show that $C$ is semiregular but, since $|G_v|=3$, it suffices to show that $\langle(y\sigma)^{2k}\rangle$ is semiregular. Since $3$ divides $k$, and $R$ is normal of index $3$ in $G$, we find that $(y\sigma)^{2k}$ is contained in the regular group $R$. This completes the proof.

\section{Preliminaries to the proof of Theorem~\ref{theorem:squarefree}}
We start with some notation and definitions. Let $G$ be a group of automorphisms of a graph $\Gamma$. We denote by $G_v$ the stabiliser in $G$ of the vertex $v$, by $\Gamma(v)$ the neighbourhood of $v$, and by  $G_v^{\Gamma(v)}$ the permutation group induced by the action of $G_v$ on $\Gamma(v)$. We say that $\Gamma$ is \emph{$G$-vertex-transitive} (\emph{$G$-arc-transitive}, respectively) if $G$ is transitive on the set of vertices (arcs, respectively) of $\Gamma$, and that it is  \emph{$G$-locally-transitive} if $G_v^{\Gamma(v)}$ is transitive for every vertex $v$.

A \emph{$t$-arc} of $\Gamma$ is a sequence of $t+1$ vertices such that any two consecutive vertices in the sequence are adjacent, and with any repeated vertices being more than $2$ steps apart.  We say that $\Gamma$ is \emph{$(G,t)$-arc-transitive} if $G$ is transitive on the set of $t$-arcs of $\Gamma$.

Given an integer $n$ and a prime $p$, we will sometimes denote by $n_p$ the $p$-part of $n$ (that is, the largest power of $p$ dividing $n$) and by $n_{p'}$ the $p'$ part (that is, $n/n_p$).

Given a graph $\Gamma$ and $N$ a group of automorphisms of $\Gamma$, the {\em quotient graph} $\Gamma/N$ is the graph whose vertices are the $N$-orbits, and with  two such $N$-orbits $v^N$ and $u^N$ adjacent whenever there is a pair of vertices $v'\in v^N$ and $u'\in u^N$ that are adjacent in $\Gamma$. If the natural projection $\pi:\Gamma\to \Gamma/N$ is a local bijection (that is, if $\pi_{|\Gamma(v)}:\Gamma(v)\to (\Gamma/N)(v^N)$  is a bijection for every vertex $v$ of $\Gamma$) then $\Gamma$ is called a \emph{regular $N$-cover} of $\Gamma/N$. Such covers have many important properties that will be used repeatedly, most of which are folklore. (See~\cite[Lemma 3.2]{GenLost} for example.)

We now collect a few  results that will be useful in the proof.
\begin{lemma}\label{pgroups}
Let $p$ be an odd prime, let $P$ be a $p$-group with a maximal cyclic subgroup and let $X$ be the group generated by elements of order $p$ in $P$.
\begin{enumerate}
\item If $P=X$, then $P$ is elementary abelian. \label{pgroup1}
\item If $X$ is cyclic, then so is $P$. \label{pgroup2}
\item If $P$ is cyclic of order at least $p^2$, then an automorphism of $P$ of order $2$ centralising the maximal subgroup of $P$ must centralise $P$. \label{pgroup3}
\end{enumerate}
\end{lemma}
\begin{proof}
These results easily follow from the classification of $p$-groups with a cyclic maximal subgroup (see for example~\cite[Section 5.3]{Stell}).
\end{proof}

\begin{lemma}\label{perfectActing}
Let $G$ be a group with a normal subgroup $N$ and let $T$ be a perfect group acting on $G$ and centralising $N$. If $T$ acts trivially on $G/N$, then it acts trivially on $G$.
\end{lemma}
\begin{proof}
Since $T$ acts trivially on $G/N$, we have $[G,T]\leq N$ and thus $[G,T,T]\leq [N,T]=1$. Similarly, $[T,G,T]=1$. By the three subgroups lemma, it follows that $[T,T,G]=1$ and, since $T$ is perfect, $[T,G]=1$.
\end{proof}

\begin{lemma}\label{IstvanLemma}
Let $\Gamma$ be a graph with every vertex having odd valency and let $C$ be a semiregular cyclic group of automorphisms of $\Gamma$. If $C$ has an odd number of orbits, then $C$ has even order and the unique involution of $C$ reverses some edge of $\Gamma$.
\end{lemma}
\begin{proof}
Let $(C_1,\ldots,C_k)$ be an ordering of the orbits of $C$ and let $A=\{a_{ij}\}$ be the $k\times k$ matrix such that $a_{ij}$ is the number of vertices of $C_j$ adjacent to a given vertex of $C_i$. It is not hard to see that this is independent of the choice of vertex, hence $A$ is well-defined and, moreover, $a_{ij}=a_{ji}$ hence $A$ is symmetric.

By hypothesis, $k$ is odd and the sum of every row and column is odd. In particular, the sum of all the entries of $A$ is odd. On the other hand, $A$ is symmetric and thus the sum of the non-diagonal entries is even. This shows that at least one diagonal entry of $A$, say $a_{nn}$, must be odd.

Let $X$ be the graph induced on $C_n$. Since $C$ is semiregular, it acts regularly on $X$ and we can view $X$ as a Cayley graph $\Cay(C,S)$. Since $X$ has odd valency, $|S|$ is odd, $|C|$ is even and $S$ contains the unique involution of $C$. The result follows.
\end{proof}

\begin{lemma}\label{lem:edge-rev}
Let $\Gamma$ be a $G$-arc-transitive group and let $N$ be a normal subgroup of $G$. If $N$ contains an element reversing some edge of $\Gamma$, then $\Gamma$ is $N$-vertex-transitive.
\end{lemma}
\begin{proof}
Let $e$ be an edge of $\Gamma$. Since $N$ is normal in the arc-transitive group $G$, $N$ must contain an element reversing $e$. In particular, the endpoints of $e$ are in the same $N$-orbit. By connectedness, $N$ is vertex-transitive.
\end{proof}

\begin{lemma}\label{IstvanIstvan}
Let $G$ be a transitive permutation group, let $N$ be a normal subgroup of $G$ and let $C$ be a semiregular subgroup of $G$ with $k$ orbits. If $|N|$ is coprime to $|G_v|$, then the induced action of $C$ on the $N$-orbits is semiregular with $k'$ orbits, where $k'$ divides $k$.
\end{lemma}
\begin{proof}
Since $|N|$ is coprime to $|G_v|$, $N_v=1$ and thus $|N|=|v^N|$ for every point $v$. It follows that $|(G/N)_{v^N}|=\frac{|G/N|}{|\Omega|/|v^N|}=\frac{|G|}{|\Omega|}=|G_v|$.

Let $c\in C$ such that $Nc$ (viewed as an element of $G/N$) fixes some $v^N$.  For the first part, it suffices to show that $cN$ is trivial. Note that, by the previous paragraph, the order of $Nc$ divides $|G_v|$. On the other hand,  since $v^N$ is fixed by $Nc$, $v^N$ can be partitioned in $\langle c\rangle$-orbits, but these all have the same size, namely $|c|$, and thus $|c|$ divides $|N|$. It follows that the order of $Nc$ divides both $|G_v|$ and $|N|$ but these are coprime and thus $Nc$ is trivial.

As for the second claim, $k=\frac{|\Omega|}{|C|}$ while $k'=\frac{|\Omega|}{|v^N|}\frac{|C\cap N|}{|C|}$ thus $\frac{k}{k'}=\frac{|v^N|}{|C\cap N|}$. Recall that $|N|=|v^N|$ and thus $\frac{k}{k'}=\frac{|N|}{|C\cap N|}$ which is an integer.
\end{proof}

\begin{lemma}\label{lemma:faithful}
Let $\Gamma$ be a $G$-arc-transitive graph. If $G$ has a normal semiregular subgroup with at most two orbits on vertices, then the subgroup of $G$ fixing a vertex and all its neighbours is trivial.
\end{lemma}
\begin{proof}
If $G$ has a normal regular group, then the result follows by~\cite[Lemma~2.1]{Godsil}. Otherwise, it is not hard to see that $\Gamma$ must be bipartite and the result follows by applying~\cite[Lemma~2.4]{Li} with $X=G$ and $N$ the bipartition-preserving subgroup of $G$.
\end{proof}

\begin{theorem}\cite{Tutte1,Tutte2}\label{TutteTheorem}
Let $\Gamma$ be a cubic graph. If $\Gamma$ is $G$-arc-transitive, then it is $(G,t+1)$-arc-regular for some $0\leq t\leq 4$. Moreover, the structure of $G_v$ is uniquely determined by $t$ and is as in Table~\ref{tab:stabs}.
\end{theorem}

\begin{table}[hhh]
\begin{tabular}{| c| ccccc|}\hline
t& 0 & 1 & 2&3&4\\ \hline
$G_v$& $\ZZ_3$ & $\Sym(3)$ & $\Sym(3)\times\ZZ_2$ & $\Sym(4)$ & $\Sym(4)\times\ZZ_2$ \\ \hline
\end{tabular}
\caption{Vertex-stabilisers in cubic $(t+1)$-arc-regular graphs}
\label{tab:stabs}
\end{table}

\begin{proposition}\cite[Corollary 4.6]{cubicPrimoz}\label{prop:yoyo}
Let $\Gamma$ be a cubic $(G,t+1)$-arc-transitive graph. If $G$ is soluble, then $t\leq 2$. Moreover, if $t=2$, then $\Gamma$ is a regular cover of  $\K_{3,3}$.
\end{proposition}

\begin{lemma}\label{coollemma}
Let $\Gamma$ be a cubic $G$-arc-transitive graph and let $N$ be a normal subgroup of $G$ that is locally-transitive on $\Gamma$. If $|N_v|\leq 12$, then $|G_v|\leq 12$.
\end{lemma}
\begin{proof}
Suppose, by contradiction, that $|G_v|> 12$. By Theorem~\ref{TutteTheorem}, $G_v$ is isomorphic to either $\Sym(4)$ or $\Sym(4)\times\ZZ_2$. Now, $N_v$ is a normal subgroup of $G_v$ of order divisible by $3$. Since $|N_v|\leq 12$, it is not hard to check that this implies $N_v\cong\Alt(4)$. Since $N_v^{\Gamma(v)}$ is a quotient of $N_v$ with order divisible by $3$, we have that $N_v^{\Gamma(v)}$ is regular of order $3$. As $N$ is normal in a vertex-transitive group, this holds for every vertex, but this implies that $N_v$ itself has order $3$, a contradiction.
\end{proof}

\begin{lemma}\label{InsolubleQuotient}
Let $\Gamma$ be a $G$-locally-transitive cubic graph. If $N$ is a normal subgroup of $G$ such that $G/N$ is insoluble, then $N$ has at least three orbits and is semiregular on the vertices of $\Gamma$. In particular, $\Gamma$ is a regular cover of $\Gamma/N$.
\end{lemma}
\begin{proof}
If $N$ has at most two orbits on vertices, then $|G:N|$ divides $2|G_v|$. Since $|G_v|$ is a $\{2,3\}$ group, so is $G/N$ and thus $G/N$ is soluble, a contradiction. If follows that $N$ has at least three orbits on vertices and, since $G$ is locally-primitive, $N$ must be semiregular.
\end{proof}

Note that, for an integer $n$, the property of having a cyclic group of index dividing $n$ is inherited by normal subgroups and quotients. This fact will be used repeatedly throughout the paper.

\begin{lemma}
\label{thm:quosolradical}
Let $\Gamma$ be a cubic $(G,t+1)$-arc-regular graph such that $G$ is insoluble and let $S$ be the soluble radical of $G$. 
\begin{enumerate}
\item If $C$ is a semiregular cyclic subgroup of $G$ with an odd number of orbits, then $|C\cap S|$ is odd and $|G/S:CS/S|_2|S|_2=2^t$. \label{quoradical1}
\item If a Sylow $2$-subgroup of $G$ has a cyclic subgroup of index at most $2^t$, then $G/S$ is almost simple. \label{quoradical2}
\end{enumerate}
\end{lemma}
\begin{proof}We first prove~\eqref{quoradical1}. Suppose, by contradiction, that $|C\cap S|$ is even. This implies that $S$ contains the unique involution of $C$. By Lemmas~\ref{IstvanLemma} and~\ref{lem:edge-rev}, it follows that $S$ is vertex-transitive, contradicting Lemma~\ref{InsolubleQuotient}. We conclude that $|C\cap S|$ is odd. Note that $|G/S:CS/S|=|G|/|CS|=|G||C\cap S|/|S||C|=3\cdot 2^tk|C\cap S|/|S|$, where $k$ is the number of orbits of $C$. Since $|C\cap S|$ is odd it follows that $|G/S:CS/S|_2|S|_2=2^t$. This concludes the proof of~\eqref{quoradical1}.

We now prove~\eqref{quoradical2}. By Theorem~\ref{TutteTheorem}, we have $0\leq t\leq 4$. By Lemma~\ref{InsolubleQuotient}, $\Gamma$ is a regular cover of $\Gamma/S$ and $G_v\cong (G/S)_{v^S}$. Let $N$ be the socle of $G/S$. Write $N=T_1\times\dots\times T_m$, such that the $T_i$'s are nonabelian simple and ordered such that the exponent of their Sylow $2$-subgroups is non-increasing. We suppose that $m\geq 2$ and will obtain a contradiction.

Let $N_2$ be a Sylow $2$-subgroup of $N$. Recall that the Sylow $2$-subgroup of a nonabelian simple group is never cyclic and, in particular, has order at least $4$. Thus, any cyclic subgroup of $N_2$ has index at least $2|T_2|_2\cdots|T_m|_2$. On the other hand, $N_2$ has a cyclic subgroup of index at most $2^t$. It follows that $2^t\geq 2|T_2|_2\cdots|T_m|_2\geq 2\cdot 4^{m-1}$. Since $t\leq 4$, we have $m=2$, $|T_2|_2\leq 8$ and $t\geq 3$.

If $N$ has at least three orbits on the vertices of $\Gamma/S$, then $\Gamma/S$ is a regular cover of $(\Gamma/S)/N$. By the Schreir Conjecture, $(G/S)/N$ is soluble and thus $t\leq 2$ by Proposition~\ref{prop:yoyo}, a contradiction. It follows that $N$ has at most two orbits. If $N$ is semiregular, then it follows by Lemma~\ref{lemma:faithful} that $t\leq 1$. We may thus assume that $N$ is locally-transitive.

We may thus apply Lemma~\ref{InsolubleQuotient} to conclude that $\Gamma/S$ is a regular cover of $(\Gamma/S)/T_1$. In particular, $N/T_1$ is locally-transitive. Since $N/T_1\cong T_2$, $|T_2|_2\leq 8$ and $(\Gamma/S)/T_1$ has even order, we find that $|(N/T_1)_{\overline{v}}|_2\leq 4$, where $\overline{v}$ is a vertex in $(\Gamma/S)/T_1$, and thus $|N_{v^S}|=|(N/T_1)_{\overline{v}}|\leq 12$. By Lemma~\ref{coollemma}, this implies $|G_v|=|(G/S)_{v^S}|\leq 12$ and thus $t\leq 2$, a contradiction.
\end{proof}

\begin{proposition}\label{allsimples}
Let $t$ be an integer with $0\leq 4\leq t$, let $k$ be a squarefree positive integer coprime to $6$ and let $\overline{G}$ be an almost simple group with order divisible by $3$. If $\overline{G}$ has a cyclic subgroup $\overline{C}$ of even order and index dividing $3\cdot 2^tk$, then $\overline{G}$, $|\overline{C}|$ and $\log_2|\overline{G}:\overline{C}|_2$ are given in Table \ref{tab:ASposs}.
\end{proposition}

\begin{table}[hhh]
\begin{tabular}{|c|l|c|c|c|c|}
\hline
&  $\overline{G}$   &  $|\overline{C}|$  & $\log_2|\overline{G}:\overline{C}|_2$ & \begin{tabular}{@{}c@{}}Upper bound \\ on $t$\end{tabular} & \begin{tabular}{@{}c@{}}Upper bound \\ on $\log_2|S|_2$\end{tabular} \\
\hline
(1)&$\Alt(5)$ &  2   & 1&1&0\\
(2)& $\Sym(5)$ &  2, 4 or 6    &1 or 2&2&1\\
(3)& $\Sym(6)$ &  6   &3&3&0\\
(4) &$\Aut(\Sym(6))$ &   6  &4&4&0\\
(5)&  $\Alt(7)$  &  6 & 2 &3&1\\
(6)& $\Sym(7)$  &  6 or 12   & 2 or 3&4&2\\
(7)& $\M_{11}$ & 6 & 3 &3&0\\
(8)& $\J_1$  &  2, 6 or 10 &    2&2&0\\
(9)&$\Aut({}^2B_2(8))$ & 4 or 12 & 4 & 0&--\\
(10)&$\PSL(2,2^4)$ & 2  & 3 &1&--\\
(11)&$\PSL(2,2^4).2$ & 2, 4, 6 or 10  & 3 or 4 & 2&--\\
(12)&$\PGammaL(2,2^4)$ & 4, 8 or 12  & 3 or 4 &2&--\\
(13)&$\PSL(2,2^5)$ & 2  & 4 &1&--\\
(14)&$\PGammaL(2,2^5)$ & 2 or 10  & 4 &1&--\\
(15)&$\PSL(2,r)$, $r\geq 7$&  $\leq(r+1)/2$   &$\geq 1$&3&2\\
(16)&$\PGL(2,r)$, $r\geq 7$ & $\leq r+1$  &$\geq 1$&3 &2 \\
(17)&$\PSigmaL(2,r^2)$, $r\geq 5$   & $2r$   & $\geq3$ &4&1\\
(18) &$\PGammaL(2,r^2)$, $r\geq 5$   & $2r$   & $\geq 4$ &4&0\\
\hline
\end{tabular}
\caption{} \label{tab:ASposs}
\end{table}

\begin{proof}
Let $T$ be the socle of $\overline{G}$. Note that $|T:T\cap \overline{C}|$  divides $3\cdot 2^tk$, this will play a crucial role.

If $T\cong\Alt(n)$, then $n< 9$ since the Sylow $3$-subgroup of $\overline{G}$ contains a cyclic subgroup of index dividing $3$. The cases $n\in\{5,6,7,8\}$ yield rows $(1-6)$ of Table~\ref{tab:ASposs}. 

Suppose now that $T$ is a sporadic simple group (including the Tits group). By considering the order of elements in  $T$ (see \cite{atlas}), one can check that $T$ does not have a cyclic subgroup of index dividing $3\cdot 2^tk$ unless $T$ is isomorphic to the Matthieu group $\M_{11}$ or the Janko group $\J_1$. Both of these have trivial outer automorphism group, hence $\overline{G}=T$ and it is easy to check that $\overline{C}$ must be as in rows (7) and (8) of Table~\ref{tab:ASposs}. 

From now on, we may thus assume that $T$ is a  simple group of Lie type, of characteristic $r$, say. We record the order and a crude upper bound on the exponent of a Sylow $r$-subgroup of $T$ in Table \ref{tab:lietype}. The orders can be found in \cite[p xvi]{atlas}, while bounds for exponents are obtained by first taking the smallest dimension $n$ of an irreducible representation of $T$ (or some covering group) over a field of characteristic $r$ from \cite[Table 5.4C]{KL},  and then using the fact that an $r$-element in $\GL(n,r^f)$ has order at most $r^e$ where $e=\lceil \log_r n\rceil \leq (n+1)/2$ (see~\cite[\S16.5]{HB} for example).

\begin{center}
\begin{table}
\begin{tabular}{|l|l|c|c|}
\hline
  $T$  &  $|T|_r$ & Upper bound on $r$-exponent & Condition   \\
\hline
$\PSL(n,r^f)$  &  $r^{fn(n-1)/2}$ &  $r^{ (n+1)/2}$& $n\geq 2$\\
$\PSU(n,r^f)$  &  $r^{fn(n-1)/2}$ & $r^{ (n+1)/2}$ & $n\geq 3$\\
$\PSp(n,r^f)$  & $r^{fn^2/4}$ &  $r^{ (n+1)/2}$& $n\geq 4$, even\\
$\POmega(n,r^f)$  & $r^{f(n-1)^2/4}$ &$r^{ (n+1)/2}$ &$n\geq 7$, $nr$ odd\\
$\POmega^\epsilon(n,r^f)$  & $r^{fn(n-2)/4}$&$r^{ (n+1)/2}$ &$n\geq 8$, $n$ even\\
$E_8(r^f)$ &$r^{120f}$ & $r^8$&\\
$E_7(r^f)$ & $r^{63f}$ & $r^6$&\\
$E_6(r^f)$&$r^{36f}$& $r^5$&\\
${}^2E_6(r^f)$ &$r^{36f}$& $r^5$&\\
$F_4(r^f)$ & $r^{24f}$& $r^5$ &\\
${}^2F_4(2^{2m+1})$  & $2^{12f}$ &  $2^5$ &$m\geq 1$\\
$G_2(r^f)$ & $r^{6f}$ & $r^2$, &$r$ odd\\
$G_2(r^f)$   & $r^{6f}$  & $r^3$, &$r=2$, $f\geq 2$\\
${}^2G_2(3^{2m+1})$  & $3^{3(2m+1)}$ &$3^2$ &$m\geq 1$\\
${}^2B_2(2^{2m+1})$ & $2^{2(2m+1)}$ & $2^2$&$m\geq 1$\\
${}^3D_4(r^f)$ & $r^{12f}$ & $r^3$ &\\
\hline
\end{tabular}
\caption{Orders and exponents of Sylow $r$-subgroups of simple groups of Lie type of characteristic $r$}
\label{tab:lietype}
\end{table}
\end{center}

Recall that $|T:T\cap \overline{C}|$  divides $3\cdot 2^tk$. In particular, a Sylow $r$-subgroup of $T$ must contain a cyclic subgroup of index at most $r$ if $r$ is odd and at most $16$ if $r=2$. Using this fact and Table~\ref{tab:lietype}, we deduce that $T$ is isomorphic to one of $\PSp(4,2)$, $\PSU(4,2)$, $\PSL(4,2)$, ${}^2B_2(8)$, $\PSU(3,r^f)$, or $\PSL(n,r^f)$ with $n\leqslant 3$.

It can be checked that $\PSU(4,2)$ and $\PSL(4,2)$ do not contain a cyclic subgroup of index dividing $3\cdot 2^tk$, whereas the case $T\cong\PSp(4,2)\cong\Sym(6)$ has already been dealt with. The group ${}^2B_2(8)$ has order coprime to $3$ but its automorphism group yields row (9) of Table~\ref{tab:ASposs}.

Suppose now that $T$ is isomorphic to $\PSL(3,r^f)$ or $\PSU(3,r^f)$. A Sylow $r$-subgroup of  $T$ has order $r^{3f}$ and exponent $2^2$ if $r=2$, and $r$ otherwise. It follows that $r=2$ and $f\leq 2$. It can be checked that no example arise when $f=2$, while $\PSU(3,2)$ is soluble. Finally, we will deal with $T\cong \PSL(3,2)\cong \PSL(2,7)$ as part of our next and last case.

It remains to deal with the case $T\cong\PSL(2,r^f)$. Since $\PSL(2,2)$ and $\PSL(2,3)$ are soluble, $\PSL(2,4)\cong\PSL(2,5)\cong\Alt(5)$ and $\PSL(2,9)\cong\Alt(6)$,  we may assume that $r^f\geq 7$ and $r^f\neq 9$. The Sylow $r$-subgroup of $T$ has order $r^f$ and exponent $r$. In particular, $f\leq 2$ unless $r=2$ in which case $f\leq 5$.

It can be checked that when $r=2$ and $f\in\{3,4,5\}$, the examples that arise are in rows $(10-14)$ of Table~\ref{tab:ASposs}. 

Suppose now that $f=1$. In particular, $r$ is odd and $\overline{G}=\PSL(2,r)$ or $\overline{G}=\PGL(2,r)$. The orders of maximal cyclic subgroups of $\PSL(2,r)$ are $(r+1)/2$, $(r-1)/2$ and $r$, while the orders of maximal cyclic subgroups of $\PGL(2,r)$ are $(r+1)$, $(r-1)$ and $r$ \cite{dickson}. Since $|\overline{C}|$ is even, we get rows (15) and (16) of Table~\ref{tab:ASposs}.

Finally, suppose that $f=2$ and $r\geq 5$. Since $k$ is squarefree and $r^2$ divides $|\PSL(2,r^2)|$, $r$ must divide $|\PSL(2,r^2)\cap\overline{C}|$. On the other hand, a Sylow $r$-subgroup $S$ of $\PSL(2,r^2)$ is elementary abelian hence $|\PSL(2,r^2)\cap\overline{C}|=r$. Moreover, for each element $c$ of order $r$ in $S$, the centraliser of $c$ in $\PGL(2,r^2)$ is $S$ \cite{dickson}. Since $|\overline{C}|$ is even, it follows that $\PSigmaL(2,r^2)\leqslant \overline{G}$ and $|\overline{C}|=2r$.  Note that $|\PSigmaL(2,r^2)|_2\geq 2^4$ and $|\PGammaL(2,r^2)|_2\geq 2^5$. This gives rows (17) and (18) of Table~\ref{tab:ASposs}.
\end{proof}

\section{Proof of Theorem~\ref{theorem:squarefree}}
In view of the statement of Theorem~\ref{theorem:squarefree}, we will consider the following hypothesis.
\begin{hypothesis}\label{hyphyp}\label{notnot}
Let $k\geq 5$ be a squarefree integer coprime to $6$ and let $\Gamma$ be a cubic $(G,t+1)$-arc-regular graph such that $C$ is semiregular with $k$ orbits.
\end{hypothesis}
Our goal is to show that $\Gamma$ has order at most $6k^2$. We introduce the following notation which we will use whenever we assume Hypothesis~\ref{hyphyp}.

\begin{notation*}
 For a prime $p$ dividing $|G|$, we denote by $P_p$ a Sylow $p$-subgroup of $G$ and by $C_p$ a Sylow $p$-subgroup of $C$ contained in $P_p$. (Note that we may have $C_p=1$.) Let $c$ be the unique involution in $C$. ($C$ has even order since $\Gamma$ does but $k$ is odd.) 

We denote by $S$ the soluble radical of $G$ and write $\overline{G}=G/S$ and $\overline{C}=CS/S$. Let $T$ be the socle of $\overline{G}$. 
\end{notation*}

 We first note a few obvious facts about $G$ and $C$ that will be very useful.

\begin{lemma}\label{basic-pty}
Assuming Hypothesis~\ref{hyphyp}, the following holds.
\begin{enumerate}
\item $|G|=3\cdot 2^tk|C|$.
\item  $|P_2:C_2|=2^t$.
\item For every odd prime $p$, we have that $|P_p:C_p|$ divides $p$.
\end{enumerate}
\end{lemma}

\subsection{$G$ Soluble}
We first focus on the case when $G$ is soluble.

\begin{lemma}\label{lemma:yii}
Assume Hypothesis~\ref{hyphyp}. If $G$ is soluble, then $t\leq 1$ and, for every prime $p$, we have $|P_p:C_p|\leq p$.
\end{lemma}
\begin{proof}
By Lemma~\ref{basic-pty}, it suffices to show that $t\leq 1$. Suppose that $t\geq 2$. By Proposition~\ref{prop:yoyo}, $t=2$ and $\Gamma$ is a regular cover of  $\K_{3,3}$.

Since $G$ is soluble, $\Gamma$ is a regular cover of $\Gamma^*$ which is itself a regular $\ZZ_q^a$-cover of  $\K_{3,3}$ for some prime $q$ and some integer $a\geq 1$. Since $t=2$, it follows by \cite[Theorem~1.1]{FengKwak2} and \cite[Theorem~4.1]{FengKwak} that $a\geq 4$, or $q=3$ and $a\neq 2$. 

By Lemma~\ref{basic-pty}, $P_q$ has a cyclic subgroup of index dividing $q$ or $4$. In particular, every elementary abelian section of $P_q$ has rank at most $2$, unless $q=2$, in which case it has rank at most $3$. By the previous paragraph, we get that $q=3$ and $a=1$ and, by \cite[Theorem~1.1]{FengKwak2}, $\Gamma^*$ is isomorphic to the Pappus graph. Since $t=2$ and the Pappus graph is $3$-arc-regular, $\Aut(\Gamma^*)$ is a quotient of $G$. This is  a contradiction because the Sylow $3$-subgroup of $\Aut(\Gamma^*)$ does not have a cyclic maximal subgroup.
\end{proof}

\begin{lemma}\label{IstvanLemma2}
Assume Hypothesis~\ref{hyphyp}, and let $p\geq 5$ be a prime dividing the order of $\Gamma$. If $P_p$ is normal in $G$ then 
\begin{enumerate}
\item $c$ does not centralise $P_p$, and \label{firstfirst}
\item $|P_p:C_p|=p$ and $|C_p|\leq p$. \label{secondsecond}
\end{enumerate}
\end{lemma}
\begin{proof}
We first prove~\eqref{firstfirst}. Suppose, by contradiction, that $c$ centralises $P_p$. Let $Z$ be the centraliser of $P_p$ in $G$. This is a normal subgroup of $G$.  By the Schur-Zassenhaus Theorem, we can write $Z=\Z(P_p) \times Y$ where $Y$ is a $p'$-group. Note that $Y$ is characteristic in $Z$ and thus normal in $G$. Since $p$ is odd, we have $c\in Y$. By Lemmas~\ref{IstvanLemma} and~\ref{lem:edge-rev}, it follows that $Y$ is transitive on the vertices of $\Gamma$, a contradiction, as $p$ divides the order of $\Gamma$. This concludes the proof of~\eqref{firstfirst}.

We now prove~\eqref{secondsecond}. By~(\ref{firstfirst}), $P_p\nleq C$ and thus Lemma~\ref{basic-pty} implies $|P_p:C_p|=p$. In particular, $P_p$ is a $p$-group with a cyclic maximal subgroup. 

Let $X$ be the group generated by elements of order $p$ in $P_p$. This is a characteristic subgroup of $P_p$ and thus normal in $G$. Suppose first that $X=P_p$. The result then follows by Lemma~\ref{pgroups}(\ref{pgroup1}).

Suppose next that $C_pX< P_p$. Since $C_p$ is maximal in $P_p$, this implies that $X\leq C_p$. It follows by Lemma~\ref{pgroups}(\ref{pgroup2}) that $P_p$ is cyclic. By~(\ref{firstfirst}), $c$ centralises $C_p$ but not $P_p$ and thus $|P_p|=p$ by Lemma~\ref{pgroups}(\ref{pgroup3}).

From now, we assume that $X<P_p=C_pX$. This implies that $1\neq P_p/X\leq CX/X$ and thus $P_p/X$ is a non-trivial normal Sylow $p$-subgroup of $G/X$. By Lemma~\ref{IstvanIstvan}, Hypothesis~\ref{hyphyp} is satisfied with $(k,\Gamma,G,C)$ replaced by $(k',\Gamma/X,G/X,CX/X)$ for some divisor $k'$ of $k$. In particular, $p$ divides the order of $\Gamma/X$ and we may apply~(\ref{firstfirst}) to conclude that $cX$ does not centralise $P_p/X$, contradicting the fact that $P_p/X\leq CX/X$.
\end{proof}

\begin{theorem}
Assume Hypothesis~\ref{hyphyp}. If $G$ is soluble, then $\Gamma$ has order at most $6k^2$.
\end{theorem}
\begin{proof}
By Lemma~\ref{lemma:yii}, every Sylow $p$-subgroup of $G$ is metacyclic. It follows by~\cite[Theorem~1]{Chillag} that $G=N\rtimes A$, where $A$ is a Hall $\{2,3\}$-subgroup of $G$ and $N$ has a normal series
$$1=N_0\norml N_1\norml \cdots\norml N_n=N$$ 
where $N_{i+1}/N_i\cong P_{p_i}$. For every $i\in\{0,\ldots,n\}$, $|N_i|$ is coprime to $6$ and thus semiregular. In particular,  $\Gamma$ is a regular cover of $\Gamma/N_i$ and, by Lemma~\ref{IstvanIstvan}, $CN_i/N_i$ is semiregular and has $\kappa_i$ orbits on $\Gamma/N_i$ for some divisor $\kappa_i$ of $k$. It follows that $(\Gamma/N_i,G/N_i)$ satisfies Hypothesis~\ref{hyphyp} with $(k,\Gamma,G,t,C)$ replaced by $(\kappa_i,\Gamma/N_i,G/N_i,t,CN_i/N_i)$. Note that $N_{i+1}/N_i$ is a normal Sylow $p_i$-subgroup of $G/N_i$ and we may thus apply Lemma~\ref{IstvanLemma2} to conclude that $|C_{p_i}|\leq |P_{p_i}:C_{p_i}|=k_{p_i}$. Finally, $k$ is coprime to $6$ but $G/N_n=G/N\cong A$ is a $\{2,3\}$-group and thus $\kappa_n=1$. Hence $\Gamma/N$ is a cubic arc-transitive circulant and thus has order at most $6$. It follows that $|C_2||C_3|\leq 6$ hence $|C|\leq 6k$, which concludes the proof.
\end{proof}

\subsection{$G$ not soluble}

We now consider the remaining case, namely when $G$ is not soluble.

\begin{lemma}\label{lem:boundS2}
Assume Hypothesis~\ref{hyphyp}. If $G$ is insoluble, then $\overline{G}$, $|\overline{C}|$,  $\log_2|\overline{G}:\overline{C}|_2$ and upper bounds for $t$ and $\log_2|S|_2$ are as in rows $(1-8)$ or $(15-18)$ of Table~\ref{tab:ASposs}.
\end{lemma}
\begin{proof}
By Lemma \ref{InsolubleQuotient}, $\Gamma/S$ is a cubic  $(\overline{G},t+1)$-arc-regular graph. In particular, $3$ divides $|\overline{G}|$. By Lemma~\ref{thm:quosolradical}, $|C\cap S|$ is odd, $\overline{C}$ is a cyclic group of even order and $\overline{G}$ is an almost simple group. Recall that $|G:C|=3\cdot 2^tk$. By Proposition~\ref{allsimples}, $\overline{G}$, $|\overline{C}|$ and $\log_2|\overline{G}:\overline{S}|$ are as in one of the rows of Table~\ref{tab:ASposs}.

We now compute upper bounds on $t$ and record them in Table~\ref{tab:ASposs}. We do this by using the fact that the isomorphism type of the vertex-stabiliser $\overline{G}_{v^S}$ is uniquely determined by $t$ (see Theorem~\ref{TutteTheorem}). For example, $\Alt(7)$ does not contain a subgroup isomorphic to $\Sym(4)\times \ZZ_2$ and thus $t\leq 3$ when $\overline{G}\cong\Alt(7)$. The fact that $\PSL(2,r)$ does not contain a subgroup isomorphic to $\Sym(4)\times \Sym(2)$ follows from Dickson's classification of the subgroups of $\PSL(2,r)$ \cite{dickson}.


We then combine this upper bound on $t$ with Lemma~\ref{thm:quosolradical}(\ref{quoradical1}) to obtain an upper bound on $\log_2|S|_2$, which we also  record in Table~\ref{tab:ASposs}. (When $\log_2|\overline{G}:\overline{C}|_2 > t$, we obtain a contradiction and record this as a --.)
\end{proof}

\begin{theorem}
Assume Hypothesis~\ref{hyphyp}. If $G$ is insoluble, then $\Gamma$ has order at most $6k^2$.
\end{theorem}

\begin{proof}
By Lemma~\ref{lem:boundS2}, $\overline{G}$, $|\overline{C}|$,  $\log_2|\overline{G}:\overline{C}|_2$ and upper bounds for $t$ and $|S|_2$ are as in Table~\ref{tab:ASposs}. Write $G=S.T.A$. Note that $\overline{G}\cong T.A$ and we can read off $A$ from Table~\ref{tab:ASposs}. In fact, $|A|\leq 2$, unless $\overline{G}\cong\PGammaL(2,r^2)$, in which case $|A|=4$. We denote by $G^\infty$ the last term of the derived series of $G$. By the Schreier conjecture, $G^\infty\cong Y.T$ for some normal subgroup $Y$ of $S$. Let $$1=S_0\norml S_1\norml \cdots\norml S_n=S$$ be a maximal characteristic series for $S$. For every $i\in\{0,\ldots,n-1\}$, let $\phi_i:G\rightarrow \Aut(S_{i+1}/S_i)$ and let $K_i$ be the kernel of $\phi_i$ in $G^\infty$.

Suppose that $\phi_i(G^\infty)$ is insoluble for some $i$. Since $S_{i+1}/S_i$ is characteristically simple, it is elementary abelian, say $S_{i+1}/S_i\cong\ZZ_p^a$. By Table~\ref{tab:ASposs}, $|S|_2\leq 4$. Together with  Lemma~\ref{basic-pty}, this implies that $a\leq 2$. Since $\phi_i(G^\infty)$ is insoluble, $a=2$, $p\geq 5$ and $\Aut(S_{i+1}/S_i)\cong\GL(2,p)$. By Dickson's classification of subgroups of $\PSL(2,p)$ \cite{dickson}, either $\SL(2,p)\leq \phi_i(G^\infty)$ or $\SL(2,5)\leq \phi_i(G^\infty)\leq\phi_i(G)\leq\SL(2,5)\circ\ZZ_{p-1}$. In the latter case, $\overline{G}\cong\Alt(5)$ and $|S|$ is even, contradicting Table~\ref{tab:ASposs}. Thus $\SL(2,p)\leqslant \phi_i(G^\infty)$ and $T=\PSL(2,p)$. By \cite[Table I]{bell}, an extension of $\ZZ_{p}^2$ by $\SL(2,p)$ splits hence $G$ contains a group of order $p^3$ and exponent $p$ as a section, contradicting Lemma~\ref{basic-pty}. 

It follows that $G^\infty/K_i\cong\phi_i(G^\infty)$ is soluble for every $i$. Since $G^\infty$ is perfect, it follows that $G^\infty=K_i$ and thus $\phi_i(G^\infty)=1$. Since this is true for every   $i\in\{0,\ldots,n-1\}$, it follows by Lemma~\ref{perfectActing} and induction that $G^\infty\leqslant C_G(S)$ and $G^\infty\cap S\leq\Z(G^\infty)$. On the other hand, $\Z(G^\infty)$ is an abelian normal subgroup of $G$ hence  $\Z(G^\infty)\leq S$ and thus $G^\infty\cap S=\Z(G^\infty)$.  In particular, $G^\infty/\Z(G^\infty)=G^\infty/(G^\infty\cap S)\cong G^\infty S/S=T$. Since $T$ is simple, we conclude that $G^\infty$ is quasisimple. In particular, $Y$ is a subgroup of the Schur multiplier of $T$. 

We want to show that the order of $\Gamma$ is at most $6k^2$. This is equivalent to $|C|\leq 6k=\frac{6|G|}{|G_v||C|}=\frac{|G|}{2^{t-1}|C|}$ and thus to $|G|\geq 2^{t-1}|C|^2$. On the other hand, $|C|=|\overline{C}||C\cap S|$ but $|C\cap S|$ is odd by Lemma~\ref{thm:quosolradical}(\ref{quoradical1}) hence $|C|\leq |\overline{C}||S|_{2'}$. Since $|G|=|\overline{G}||S|\geq |\overline{G}||S|_{2'}$, it thus suffices to show that 
\begin{equation}\label{coolnewbound}
|\overline{G}|\geq 2^{t-1}|\overline{C}|^2|S|_{2'}.
\end{equation}

Suppose first that $G^\infty$ is semiregular on the vertices of $\Gamma$. This implies that $|G/G^\infty|_2\geq 2^t$. On the other hand, $G/G^\infty\cong (S/Y).A$ hence $|G/G^\infty|_2\leq |S|_2|A|$.  Combining this with Lemma~\ref{thm:quosolradical}(\ref{quoradical1}), we get $|G/G^\infty|_2|\overline{G}:\overline{C}|_2\leq 2^t|A|\leq |G/G^\infty|_2|A|$ and thus $|\overline{G}:\overline{C}|_2\leq |A|$. By running through Table~\ref{tab:ASposs}, we find that $\overline{G}\cong\PGL(2,r)$ with $r\geq 5$ and $|A|=|\overline{G}:\overline{C}|_2=2$. (Note that this includes the case $\overline{G}\cong\Sym(5)$.) Using the previous inequalities, this implies that $|G/G^\infty|_2= 2^t$ and thus $G^\infty$ has an odd number of orbits. Since $G^\infty$ is semiregular, it follows that it is transitive hence $|(S/Y).A|=|G/G^\infty|=|G_v|=3\cdot 2^t$ and Lemma~\ref{lemma:faithful} implies $t\leq 1$. Since $|A|=2$, we have $t=1$ and $|S/Y|=3$. On the other hand, since the Schur multiplier of $\PSL(2,r)$ has order $2$, we see that $|Y|_{2'}=1$ and thus $|S|_{2'}=3$. Now, $|\overline{G}|=(r+1)r(r-1)$ while $|\overline{C}|\leq r+1$ hence~\eqref{coolnewbound} is satisfied.

We may thus assume that $G^\infty$ is not semiregular on the vertices of $\Gamma$. In particular, $G^\infty$ is locally transitive and has at most two orbits on the vertices of $\Gamma$. It follows that $G/G^\infty$ is a $2$-group hence so is $S/Y$ and thus $|S|_{2'}=|Y|_{2'}$. Now, by considering the Schur multiplier of $T$, we find that $|Y|_{2'}=1$ unless $T$ is isomorphic to  $\Alt(6)$ or $\Alt(7)$, when we may have $|Y|_{2'}=3$. It is then a matter of routine to go through Table~\ref{tab:ASposs} and verify that~\eqref{coolnewbound} is satisfied.
\end{proof}


%
%
%

\noindent\textsc{Acknowledgements.}
We would like to thank Gordon Royle for his help with the computations mentioned at the end of the introduction, and Luke Morgan for pointing out Lemma~\ref{perfectActing}.

\end{document}